\documentclass{article}
\usepackage{amssymb}
\usepackage{amsmath}
\usepackage{amsfonts}
\usepackage{latexsym}
\usepackage{amsthm}
\usepackage[colorlinks]{hyperref}

\newtheorem{theorem}{Theorem}

\newtheorem{corollary}[theorem]{Corollary}

{\theoremstyle{definition}
\newtheorem{definition}{Definition}[section]
}
{\theoremstyle{definition}
\newtheorem{remark}{Remark}[section]
}
\newtheorem{lemma}[theorem]{Lemma}

\newtheorem{proposition}[theorem]{Proposition}

\input{tcilatex}
\begin{document}

\title{{\LARGE CLOSED CONVEX SETS OF MOTZKIN AND GENERALIZED MINKOWSKI TYPES}%
}
\author{\textsc{JUAN ENRIQUE MART\'{I}NEZ LEGAZ}\thanks{%
Corresponding author}\thanks{%
This author was supported by the Ministerio~de Ciencia, Innovaci\'{o}n~y
Universidades~[PGC2018-097960-B-C21] and the Severo Ochoa Programme for
Centres of Excellence in R\&D [SEV-2015-0563]. He is affiliated with MOVE
(Markets, Organizations and Votes in Economics).} \\
%EndAName
Dept. of Economics and Economic History\\
Universitat Aut\`{o}noma de Barcelona,\\
and BGSMath\\
Spain \and \textsc{CORNEL PINTEA}\thanks{%
This author was supported by Ministerio~de Ciencia, Innovaci\'{o}n~y
Universidades~[PGC2018-097960-B-C21] and by AGC31333/23.03.2022 and
AGC31334/23.03.2022.} \\
%EndAName
Babe\c{s}-Bolyai University\\
Faculty of Mathematics and Computer Science\\
Cluj-Napoca\\
Romania}
\date{\textit{Dedicated to Francis Clarke on the occasion of his 75th
birthday}}
\maketitle

\begin{abstract}
\noindent The aim of this paper is twofold. On one hand the generalized
Minkowski sets are defined and characterized. On the other hand, the Motzkin
decomposable sets, along with their epigraphic versions are considered and
characterized in new ways. Among them, the closed convex sets with one
single minimal face, i.e. translated closed convex cones, along with their
epigraphic counterparts are particularly studied.
\end{abstract}



\section{Introduction}

Every compact convex subset of the Euclidean space is, according to the
Minkowski theorem \cite{Mi}, the convex hull of the set of its extreme
points (see also \cite[p.52]{Barvinok} and \cite[Theorem 2.7.2]{BoVa}).
However, the class of all closed convex subsets that are representable as
their own sets of extreme points is significantly much larger and we called
them \emph{Minkowski sets}. Indeed, there are many unbounded Minkowski sets
and we characterized them in \cite{MaLe-Pi}. For such subsets the extreme
points are obviously minimal and lowest dimensional faces as well. In this
paper we first prove that a face of a closed convex subset of the Euclidean
space is minimal if and only if it is lowest dimensional. Such faces are
also shown to be lineality invariant, and several other properties of them
are proved. We next enlarge the class of Minkowski sets to the one of \emph{%
generalized Minkowski sets}, which are intensively studied here. This latter
class consists of all closed convex subsets of the Euclidean space that are
representable as the convex hull of the union of its minimal (lowest
dimensional) faces.

The paper is organized as follows: In the second section we prove several
characterizations of the lowest dimensional faces of a closed convex subset
of the Euclidean space. This is done by proving several helpful statements
before. The equivalence between the minimality of a face of a closed convex
set and its quality to be lowest dimensional appears as a corollary
afterwards. The third section is devoted to new characterizations of Motzkin
decomposable sets and their epigraphic versions. Note that several
characterizations of Motzkin decomposable sets were done before by Goberna,
Martinez-Legaz and Todorov \cite{Go-MaLe-To}. The particular case of
translated closed convex cones is also studied by highlighting their
property to have one single minimal face. Actually, this property
characterizes closed convex cones. Finally, in the last section an important
class of generalized Minkowski sets is highlighted, namely that consisting
of those closed convex sets whose Pareto like sets cover their relative
boundaries. \label{intro}

\section{Definitions and preliminary results}

We consider the closed convex subsets of $\mathbb{R}^{n}$ that can be
represented as the convex hull of their lowest dimensional faces. The
Minkowski sets (see \cite{MaLe-Pi}) are obviously particular examples of
such closed convex sets. On the other hand, every Minkowski set $C$ can
produce extra sets with the above property simply by considering $C+U$,
where $U$ is a suitable subspace. We will prove that, in fact, every set
with the above property is the sum of a Minkwoski set with a subspace.

Thoroughout the whole paper, $0^{+}C:=\{y\in \mathbb{R}^{n}:y+C\subseteq C\}$
stands for the \emph{recession cone} of a closed convex set $C$. Recall that 
$0^{+}C$ is a convex cone \cite[Theorem 18.1]{R70} and $\mathrm{lin}%
~C:=0^{+}C\cap (-0^{+}C)$ is a subspace of $\mathbb{R}^{n}$ called \emph{the
lineality} of $C$.

Recall that the \emph{polar con}e of a set $S\subseteq \mathbb{R}^{n}$ is%
\begin{equation*}
S^{0}:=\left\{ x^{\ast }\in \mathbb{R}^{n}:\left\langle x^{\ast
},x\right\rangle \leq 0\text{\qquad }\forall x\in S\right\} ,
\end{equation*}%
and the \emph{barrier cone} $\mathrm{barr}(C)$ of $C$ consists of those
vectors $x^{\ast }\in \mathbb{R}^{n}$ for which there exists $\alpha
_{x^{\ast }}\in \mathbb{R}$ with the property $\langle c,x^{\ast }\rangle
\leq \alpha _{x^{\ast }}$, for every $c\in C$. In other words,%
\begin{equation*}
\mathrm{barr}(C):=\left\{ x^{\ast }\in \mathbb{R}^{n}\ |\ \sup_{c\in
C}\langle c,x^{\ast }\rangle <\infty \right\} .
\end{equation*}%
It is well known that $\mathrm{barr}(C)$ is a convex cone and $[\mathrm{barr}%
(C)]^{\circ }=0^{+}C$.

\begin{definition}
A vector $x^*$ is said to be \emph{normal} to a convex set $C\subseteq%
\mathbb{R}^{n}$ at a point $x\in C$ if $\langle c-x,x^*\rangle\leq 0$, for
all $c\in C$. The set of vectors normal to $C$ at $x$ is a closed convex
cone denoted by $N_C(x)$ and is called the \emph{normal cone} to $C$ at $x$.
\end{definition}

Note that the normal cone to $C$ at every interior point of $C$ reduces to $%
\{0\}$. We extend the, possibly multivalued, mapping $N_{C}:C%
\rightrightarrows \mathbb{R}^{n}$ to the whole space $\mathbb{R}^{n}$ by
setting $N_{C}(x):=\emptyset $ whenever $x\in \mathbb{R}^{n}\setminus C$.
Denote by $N_{C}(\mathbb{R}^{n})$ the range of $N_{C}$, i.e. the union of
the normal cones to $C$ at all points of $C,$ and call it the \emph{total
normal cone} of $C$. Note that the total normal cone need not be convex.

For a lower semicontinuous proper convex function $f:\mathbb{R}%
^{n}\rightarrow \mathbb{R\cup }\left\{ +\infty \right\} ,$ one has $\partial
f^{\ast }=(\partial f)^{-1}$, i.e. 
\begin{equation}
x^{\ast }\in \partial f(x)\Leftrightarrow x\in \partial f^{\ast }(x^{\ast }),
\label{equivalence}
\end{equation}%
where $f^{\ast }$ stands for the conjugate function of $f$ and $\partial
f(x) $ for the subdifferential of $f$ at $x$. Recall that $f^{\ast }$ and $%
\partial f$ are defined by 
\begin{align*}
& f^{\ast }(x^{\ast }):=\sup_{x}\{\langle x,x^{\ast }\rangle -f(x)\}, \\
& \partial f(x):=\{x^{\ast }\ |\ f(y)\geq f(x)+\langle x^{\ast },y-x\rangle
,\forall y\}.
\end{align*}

For example, the support function $\sigma _{_{\!C}}$ of $C$ is the conjugate 
$\delta _{_{\!C}}^{\ast }$ of the indicator function $\delta _{_{\!C}}$ of $%
C $. Recall that these two functions are defined by

\begin{align*}
& \sigma _{_{\!C}}:\mathbb{R}^{n}\longrightarrow \mathbb{R}^{n},\ \sigma
_{_{\!C}}(x^{\ast }):=\sup_{_{\!c\in C}}\langle c,x^{\ast }\rangle , \\
& \delta _{_{\!C}}(c):=\left\{ 
\begin{array}{ll}
0 & \mbox{ if }c\in C \\ 
+\infty & \mbox{ if }c\not\in C.%
\end{array}%
\right.
\end{align*}%
We also observe that 
\begin{equation*}
\mathrm{dom}(\sigma _{_{\!C}}):=\{x^{\ast }\in \mathbb{R}^{n}\ |\ \sigma
_{_{\!C}}(x^{\ast })<+\infty \}=\mathrm{barr}(C).
\end{equation*}%
and 
\begin{equation*}
\mathrm{dom}(\partial \sigma _{_{\!C}}):=\{x^{\ast }\in \mathbb{R}^{n}\ |\
\partial \sigma _{_{\!C}}\neq \emptyset \}=N_{C}(\mathbb{R}^{n}),
\end{equation*}%
as 
\begin{equation*}
\partial \delta _{_{\!C}}=N_{C};
\end{equation*}%
moreover, as a consequence of (\ref{equivalence}), if $C$ is convex and
closed, then 
\begin{equation*}
x^{\ast }\in N_{C}(x)\Leftrightarrow x\in \partial \sigma _{_{\!C}}(x^{\ast
}).
\end{equation*}

\begin{theorem}
\label{prop05.06.17.1}Let $C\subseteq \mathbb{R}^{n}$ be a nonempty closed
convex set, $U\subseteq \mathbb{R}^{n}$ be a supplementary subspace to $%
\mathrm{lin}~C$ (i.e., a linear subspace such that $U\oplus \mathrm{lin}~C=%
\mathbb{R}^{n}$), and $F$ be a nonempty face of $C$. The following
statements are equivalent:

\begin{enumerate}
\item $F$ is lowest dimensional.\label{st05.06.17.1}

\item $F$ is an affine variety.\label{st05.06.17.2}

\item $F=x+\mathrm{lin}~C$ for some $x\in C.$\label{st05.06.17.3 copy(1)}

\item $F\cap U$ is a singleton.\label{st05.06.17.4}

\item $F\cap U$ is the singleton of an extreme point of $C\cap U.$\label%
{st05.06.17.5}

\item $F=x+\mathrm{lin}~C$ for some $x\in \mathrm{ext}\left( C\cap U\right) $%
.\label{st05.06.17.6}
\end{enumerate}
\end{theorem}

\begin{remark}
\label{min faces}In fact, 6 implies that $F$ is a face of $C.$ To see this,
we consider two points $p,q\in C$ such that $(1-t)p+tq\in x+\mathrm{lin}~C$
for some $t\in (0,1)$ and the decomposition $p=p_{U}+p_{l}$ and $%
q=q_{U}+q_{l}$, where $p_{U},q_{U}\in U$ and $p_{l},q_{l}\in \mathrm{lin}~C$%
, due to the direct sum decomposition $U\oplus \mathrm{lin}~C=\mathbb{R}^{n}$%
. Note that $p_{U}=p-p_{l},\ q_{U}=q-q_{l}\in C+\mathrm{lin}~C=C$, namely $%
p_{U},q_{U}\in C\cap U,$ and therefore 
\begin{equation*}
(1-t)p_{U}+tq_{U}=(1-t)p+tq-((1-t)p_{l}+tq_{l})\in x+\mathrm{lin}~C-\mathrm{%
lin}~C=x+\mathrm{lin}~C.
\end{equation*}%
Consequently,%
\begin{equation*}
(1-t)p_{U}+tq_{U}\in C\cap U\cap (x+\mathrm{lin}~C)\subseteq U\cap (x+%
\mathrm{lin}~C)=\{x\}
\end{equation*}%
(the latter equality following from the fact that $x\in U$), which shows
that $p_{U},q_{U}=x$, as $x$ is an extreme point of $C\cap U$. Thus%
\begin{equation*}
p=x+p_{l}\text{, }q=x+q_{l}\in x+\mathrm{lin}~C,
\end{equation*}%
which completes the proof that $F$ is a face of $C$.
\end{remark}

\begin{corollary}
A face $F$ of the closed convex set $C\subseteq \mathbb{R}^{n}$ is lowest
dimensional if and only if $F$ is minimal.
\end{corollary}

\begin{proof}
If $F$ is lowest dimensional, then $F$ is obviously minimal, as a proper
face of $F$ would have dimension strictly smaller than $\dim ~F$.
Conversely, assume that $F$ is minimal and not lowest dimensional, i.e. $F$
is not an affine variety, according to Theorem \ref{prop05.06.17.1}. Since $%
F $ is minimal, it is not a closed half of any affine variety either, as
every halfflat has a proper face. According to \cite[Theorem 2.6.12]{Web}, $%
F=\mathrm{conv}~\mathrm{rbd}~F$ and the inequality $\dim ~F_{a}<\dim ~F$
holds for every $a\in \mathrm{rbd}~F$, where $F_{a}$ is the smallest face of 
$F$ containing $a$, i.e. the intersection of all faces of $F$ that contain $%
a $ (see \cite[Corollary 2.6.11]{Web}). This shows that $F$ is not minimal,
as it has a proper face $F_{a}$. We have thus obtained a contradiction.
\end{proof}

\bigskip

We will denote by $MF\left( C\right) $ the union of the minimal
(equivalently, lowest dimensional) faces of a closed convex set $C\subseteq 
\mathbb{R}^{n}.$

\begin{corollary}
\label{LDF}Let $C\subseteq \mathbb{R}^{n}$ be a nonempty closed convex set
and $U\subseteq \mathbb{R}^{n}$ be a supplementary subspace to $\mathrm{lin}%
~C$ . Then%
\begin{equation*}
MF\left( C\right) =\mathrm{ext}(C\cap U)+\mathrm{lin}~C.
\end{equation*}

In particular,%
\begin{equation*}
MF\left( C\right) =\mathrm{ext}(C\cap \left( \mathrm{lin}~C\right) ^{\bot })+%
\mathrm{lin}~C.
\end{equation*}
\end{corollary}

In order to prove Theorem \ref{prop05.06.17.1}, we need some preliminary
results.

\begin{proposition}
\label{rem19.08.17}If $F$ is a nonempty face of a closed convex set $%
C\subseteq \mathbb{R}^{n},$ then%
\begin{equation*}
\mathrm{lin}~F=\mathrm{lin}~C.
\end{equation*}
\end{proposition}

\begin{proof}
The inclusion $\subseteq $ is obvious, as one has $0^{+}(F)\subseteq
0^{+}(C) $ (see \cite[Theorem 8.3, p. 63]{R70}). To prove the opposite
inclusion, let $d\in \mathrm{lin}~C$ and take $x\in F.$ Since $x=\frac{1}{2}%
\left( x+d+x-d\right) $ and $x+d,$ $x-d\in C,$ we have $x+d,$ $x-d\in F,$
which shows that both $d$ and $-d$ belong to $0^{+}(F),$ that is, $d\in 
\mathrm{lin}~F,$ and the equality $\mathrm{lin}~F=\mathrm{lin}~C$ is
completely proved.
\end{proof}

\begin{corollary}
\label{prop24.02.15.1}If $F$ is a nonempty face of a closed convex set $%
C\subseteq \mathbb{R}^{n}$, then $F$ contains $l_{_{C}}$-dimensional affine
varieties, where $l_{_{C}}=\dim (\mathrm{lin}~C)$. In particular $\dim F\geq
\dim (\mathrm{lin}~C)$.
\end{corollary}

\begin{lemma}
\label{Lemma25.02.15.1}Let $C\subseteq \mathbb{R}^{n}$ be a nonempty closed
convex set and $U\subseteq \mathbb{R}^{n}$ be a supplementary subspace to $%
\mathrm{lin}~C.$ A nonempty face $F$ of $C$ is lowest dimensional if and
only if it is an $l_{_{C}}$-dimensional affine variety of $\mathbb{R}^{n}$
parallel with $\mathrm{lin}~C$. In such a case the intersection $F\cap U$ is
a singleton, say $\left\{ x_{_{F}}\right\} $, with $x_{_{F}}$ being an
extreme point of $C\cap U,$ and $F=x_{_{F}}+\mathrm{lin}~C$.
\end{lemma}

\begin{proof}
Assume that $F$ is a lowest dimensional face of the closed convex set $C$.
Then $F$ has no proper faces and empty relative boundary therefore. In other
words the boundary of $F$ in $\mathrm{aff}(F)$ is empty, which shows, due to
the connectedness of $\mathrm{aff}(F)$, that $F=\mathrm{aff}(F)$ (see e.g 
\cite[p. 86]{VINK}). Thus $F=x+V$ for some $x\in F$ and some subspace $V$ of 
$\mathbb{R}^{n}$. Since%
\begin{equation*}
F+V=x+V+V=x+V=F,
\end{equation*}%
we have $V\subseteq 0^{+}C$ as well as $V=-V\subseteq -0^{+}C$, namely%
\begin{equation*}
V\subseteq 0^{+}C\cap (-0^{+}C)=\mathrm{lin}~C.
\end{equation*}%
To prove the converse inclusion, let $d\in \mathrm{lin}~C.$ Then we have $%
x+d,x-d\in C;$ hence, since $\frac{1}{2}\left( x+d\right) +\frac{1}{2}\left(
x-d\right) \in F$ and $F$ is a face, we have $x+d\in F=x+V,$ that is, $d\in
V.$ This shows that $V=\mathrm{lin}~C$.

Conversely, if the face $F$ of $C$ is an $l_{_{C}}$-dimensional affine
variety of $\mathbb{R}^{n}$ parallel with $\mathrm{lin}~C$, then it is
obviously lowest dimensional due to Corollary \ref{prop24.02.15.1}.

The intersection $F\cap U$ is obviously a singleton, say $\left\{
x_{_{F}}\right\} $, and one has $F=x_{_{F}}+lin~C$. In order to justify the
extreme property of $x_{_{F}}$, assume that $x_{_{F}}=(1-t)p+tq$ for some $%
p,q\in C\cap U$ and some $t\in ]0,1[$. Since $x_{_{F}}\in F$ and $F$ is a
face of $C$, it follows that $p,q\in F=x_{_{F}}+\mathrm{lin}~C$. Thus $%
p=x_{_{F}}+p_{l}$ and $q=x_{_{F}}+q_{_{l}}$ for some $p_{_{l}},\ q_{_{l}}\in 
\mathrm{lin}~C$. On the other hand,%
\begin{equation*}
p_{_{l}}=p-x_{_{F}},\ q_{_{l}}=q-x_{_{F}}\in U.
\end{equation*}%
as $p,q,x_{_{F}}\in U$. Thus $p_{_{l}},q_{_{l}}\in U\cap \mathrm{lin}~C\cap
U $, which shows that $p_{_{l}}=q_{_{l}}=0$ and therefore $p=x_{_{F}}=q$.
\end{proof}

\bigskip

\begin{proof}[Proof of Theorem \protect\ref{prop05.06.17.1}]
The equivalence [1 $\Leftrightarrow $ 3] and the implication [1 $%
\Longrightarrow $ 5] follow from Lemma \ref{Lemma25.02.15.1}.

[2 $\Longrightarrow $ 3]. Assume that $F=x+V$ for some $x\in F$ and a
subspace $V$ of $\mathbb{R}^{n}$. Since $V=\mathrm{lin}~F,$ equality 3
follows via Proposition \ref{rem19.08.17}.

The implications [3 $\Longrightarrow $ 2], [5 $\Longrightarrow $ 4] and [6 $%
\Longrightarrow $ 3] are obvious.

The implication [4 $\Longrightarrow $ 6] follows from the equality $F=F\cap
U+\mathrm{lin}~F.$
\end{proof}

\section{Closed convex sets with a single minimal face and Motzkin
decomposable sets}

In this section we first characterize the closed convex translated cones
along with their epigraphic counterpats in terms of their minimal faces. The
Motzkin decomposable sets along with the Motzkin decomposable functions are
characterized afterwards.

\begin{remark}
\label{rem20.11.22.1}A closed convex cone $K$ has one single minimal face.
Its single minimal face is $\mathrm{lin}~K$. Indeed we first observe that $0$
is the only extreme point of $K\cap (\mathrm{lin}~K)^{\perp }$. Therefore $%
\mathrm{lin}~K=0+\mathrm{lin}~K$ is the only minimal face of $K$, due to
Theorem \ref{prop05.06.17.1}\eqref{st05.06.17.6}. Note that a translated
closed convex cone has the same property.
\end{remark}

\begin{proposition}
\label{lemma20.11.22.1}If a closed convex set $C\subseteq \mathbb{R}^{n}$
has one single minimal face $F,$ then this face is contained in any face of $%
C$. In other words, $F$ is the intersection of all faces of $C$.
\end{proposition}

\begin{proof}
It is an immediate consequence of the fact that every face of $C$ contains a
minimal face.
\end{proof}

\begin{proposition}
\label{prop20.11.22.1}A nonempty closed convex set $C\subseteq \mathbb{R}%
^{n} $ has one single minimal face if and only if it is a closed convex cone
or a translated closed convex cone.
\end{proposition}

\begin{proof}
The obvious implication is Remark \ref{rem20.11.22.1}. We now assume that
the closed convex set $C$ has one single minimal face. We assume that the
origin of $\mathbb{R}^{n}$ belongs to the minimal face $F$ of $C$, as
otherwise we translate the set with the opposite of a vector in $F$. We will
prove that, under this assumption, $C$ is actually a cone. In what follows
we will use an inductive argument on the dimension of $C$ to show that the
positive multiples of the vectors in $C$ remain in $C$. Indeed, if $C$ is $0$%
-dimensional, i.e. $C=\left\{ 0\right\} $, then it is obviously a cone. We
now assume that $\dim C\geq 1$ and the proper faces of $C$, which have a
strictly smaller dimension, are cones. If $x\in C$ would be a point such
that $tx\not\in C$ for some $t>0$, then $sx\in \mathrm{rbd}~C$, where $%
s=\sup \{\lambda >0\ |\ \lambda x\in C\}\geq 1$, i.e $sx$ belongs to a
proper face of $C$, which is, by the inductive hypothesis, a cone. In
particular, $tx=ts^{-1}(sx)$ belongs to that proper face and therefore to $C$%
, which is absurd.
\end{proof}

\begin{remark}
\label{epigraph cone}The epigraph of a function $f:\mathbb{R}%
^{n}\longrightarrow \mathbb{R\cup }\left\{ +\infty \right\} $ is a nonempty
closed convex cone if and only if $f$ is in the orbit of a lower
semicontinuous (l.s.c., in brief) proper sublinear function $\phi :\mathbb{R}%
^{n}\longrightarrow \mathbb{R\cup }\left\{ +\infty \right\} $ with respect
to the action 
\begin{equation*}
(\mathbb{R}^{n}\times \mathbb{R})\times \mathbb{R}^{\mathbb{R}%
^{n}}\longrightarrow \mathbb{R}^{\mathbb{R}^{n}},\left( (u,v),\phi \right)
\mapsto (u,v)\oplus \phi ,
\end{equation*}%
where $((u,v)\oplus \phi )(x):=\phi (x+u)-v$. Note that 
\begin{equation*}
\mathrm{epi}\left( (u,v)\oplus \phi \right) =\mathrm{epi}~\phi -(u,v),
\end{equation*}%
for all $\left( (u,v),\phi \right) \in (\mathbb{R}^{n}\times \mathbb{R}%
)\times \mathbb{R}^{\mathbb{R}^{n}}$.
\end{remark}

\bigskip

From Remark \ref{epigraph cone} and Proposition \ref{prop20.11.22.1}, one
obtains the following result.

\begin{corollary}
Let $f:\mathbb{R}^{n}\longrightarrow \mathbb{R\cup }\left\{ +\infty \right\} 
$ be a l.s.c. proper convex function. Then there exist $u\in \mathbb{R}^{n}$
and $v\in \mathbb{R}$ such that $f(\cdot +u)-v$ is sublinear if and only if $%
\mathrm{epi}~f$ has one single minimal face.
\end{corollary}

The last result in this section is a counterpart to Corollary \ref{char
cones} for functions.

\begin{proposition}
For a l.s.c. proper convex function $f:\mathbb{R}^{n}\longrightarrow \mathbb{%
R\cup }\left\{ +\infty \right\} $ and a point $u\in \mathrm{dom}~f$ such
that $\mathrm{dom}~f-u$ is a cone, there exists $v\in \mathbb{R}$ such that $%
f(\cdot +u)-v$ is sublinear if and only if $\partial f\left( \mathbb{R}%
^{n}\right) =\partial f\left( u\right) .$
\end{proposition}

\begin{proof}
Let $g:=f(\cdot +u)-v.$ Since $\mathrm{dom}~g=\mathrm{dom}~f-u$ and $%
\partial g\left( x\right) =\partial f\left( x+u\right) $ for every $x\in 
\mathbb{R}^{n},$ we assume, without loss of generality, that $u=0$ and $v=0,$
so that $\mathrm{dom}~f$ is a cone and $f\left( 0\right) =0.$ If $f$ is
sublinear, from the well known (and easy to prove) equality%
\begin{equation}
\partial f\left( x\right) =\partial f\left( 0\right) \cap \left\{ x^{\ast
}\in \mathbb{R}^{n}:\left\langle x^{\ast },x\right\rangle =f\left( x\right)
\right\} ,  \label{subd sublinear}
\end{equation}%
which holds for every $x\in \mathbb{R}^{n},$ it immediately follows that $%
\partial f\left( \mathbb{R}^{n}\right) =\partial f\left( 0\right) .$
Conversely, assume that the latter equality holds and define $h:\mathbb{R}%
^{n}\longrightarrow \mathbb{R\cup }\left\{ +\infty \right\} $ by $h\left(
x\right) :=\sup_{x^{\ast }\in \partial f\left( 0\right) }\left\langle
x^{\ast },x\right\rangle .$ Clearly, $h$ is sublinear and, since $f\left(
0\right) =0,$ it is a minorant of $f.$ We will actually prove that $h=f.$ To
this aim, let $x^{\prime }\in \mathrm{dom}~\partial f$ and take $x_{0}^{\ast
}\in \partial f\left( x^{\prime }\right) .$ Then, by (\ref{subd sublinear}),
we have $f\left( x^{\prime }\right) =\left\langle x_{0}^{\ast },x^{\prime
}\right\rangle .$ Hence, by the lower semicontinuity of $f$ and the fact
that $\mathrm{dom}~\partial f$ is dense in $\mathrm{dom}~f$ (by Br\o %
ndsted--Rockafellar theorem), for every $x\in \mathrm{dom}~f$ we get $%
f\left( x\right) \leq \left\langle x_{0}^{\ast },x\right\rangle \leq
\sup_{x^{\ast }\in \partial f\left( 0\right) }\left\langle x^{\ast
},x\right\rangle =h\left( x\right) ,$ the latter inequality following from
the relations $x_{0}^{\ast }\in \partial f\left( x\right) \subseteq \partial
f\left( \mathbb{R}^{n}\right) =\partial f\left( 0\right) .$ We have thus
proved that $f=h$ on $\mathrm{dom}~f.$ Since $\mathrm{dom}~f$ is a convex
cone, this shows that $f$ is sublinear.
\end{proof}

\bigskip We say that a nonempty closed convex set $C\subseteq \mathbb{R}^{n}$
is Motzkin decomposable ($M$-decomposable in short) if there exist a
decomposition $C=C_{0}+K,$ with the set $C_{0}\subseteq \mathbb{R}^{n}$
being convex and compact and $K$ being a closed convex cone. Then we say
that $C_{0}$ and $K$ are the compact and conic components, respectively, of
that decomposition. The conic component is uniquely determined, namely, one
has $K=0^{+}\left( C\right) .$ The original Motzkin Theorem \cite{Mo}
asserts that every polyhedral convex set is the sum of a polytope and a
polyhedral convex cone, namely it is $M$-decomposable. From the following
characterization of Motzkin decomposable sets, we will obtain a new
characterization of closed convex cones.

\begin{proposition}
\label{char Motzk}For a nonempty closed convex set $C\subseteq \mathbb{R}%
^{n} $ and a compact convex set $C_{0}\subseteq C,$ the following statements
are equivalent:

\begin{enumerate}
\item $C$ is Motzkin decomposable and $C_{0}$ is a compact component of $C.$

\item $N_{C}\left( \mathbb{R}^{n}\right) =N_{C}\left( \mathbb{C}_{0}\right)
. $
\end{enumerate}
\end{proposition}

\begin{proof}
\lbrack 1 $\Longrightarrow $ 2] It is enough to observe that, for $c\in
C_{0} $ and $d\in 0^{+}\left( C\right) ,$ one has $N_{C}\left( c+d\right)
\subseteq N_{C}\left( c\right) .$

[2 $\Longrightarrow $ 1] From the well known equality%
\begin{equation*}
C=\left\{ x\in \mathbb{R}^{n}:\left\langle x^{\ast },x\right\rangle \leq
\delta _{C}^{\ast }\left( x^{\ast }\right) \text{\qquad }\forall x^{\ast
}\in N_{C}\left( \mathbb{R}^{n}\right) \right\} ,
\end{equation*}%
using that $\delta _{C}^{\ast }\left( x^{\ast }\right) =\delta
_{C_{0}}^{\ast }\left( x^{\ast }\right) $ for every $x^{\ast }\in
N_{C}\left( \mathbb{C}_{0}\right) ,$ by 2 we obtain%
\begin{equation}
C=\left\{ x\in \mathbb{R}^{n}:\left\langle x^{\ast },x\right\rangle \leq
\delta _{C_{0}}^{\ast }\left( x^{\ast }\right) \text{\qquad }\forall x^{\ast
}\in N_{C}\left( \mathbb{C}_{0}\right) \right\} .  \label{expr for C}
\end{equation}%
Since $\delta _{C_{0}}^{\ast }$ is continuous, we can replace $N_{C}\left( 
\mathbb{C}_{0}\right) $ with its closure in (\ref{expr for C});\ therefore,
by $cl~N_{C}\left( C_{0}\right) =cl~N_{C}\left( \mathbb{R}^{n}\right)
=cl~barr~C=\left( barr~C\right) ^{00}=\left( 0^{+}\left( C\right) \right)
^{0},$ we obtain%
\begin{equation*}
C=\left\{ x\in \mathbb{R}^{n}:\left\langle x^{\ast },x\right\rangle \leq
\delta _{C_{0}}^{\ast }\left( x^{\ast }\right) \text{\qquad }\forall x^{\ast
}\in \left( 0^{+}\left( C\right) \right) ^{0}\right\} .
\end{equation*}%
Observe that this equality can be equivalently written%
\begin{equation}
C=\left\{ x\in \mathbb{R}^{n}:\left\langle x^{\ast },x\right\rangle \leq
\delta _{C_{0}}^{\ast }\left( x^{\ast }\right) +\delta _{\left( 0^{+}\left(
C\right) \right) ^{0}}\left( x^{\ast }\right) \text{\qquad }\forall x^{\ast
}\in \mathbb{R}^{n}\right\} .  \label{expre for C 2}
\end{equation}%
Hence, since $\delta _{\left( 0^{+}\left( C\right) \right) ^{0}}=\delta
_{0^{+}\left( C\right) }^{\ast }$ and $\delta _{C_{0}}^{\ast }+\delta
_{0^{+}\left( C\right) }^{\ast }=\delta _{C_{0}+0^{+}\left( C\right) }^{\ast
},$ equality (\ref{expre for C 2}) yields%
\begin{equation*}
C=\left\{ x\in \mathbb{R}^{n}:\left\langle x^{\ast },x\right\rangle \leq
\delta _{C_{0}+0^{+}\left( C\right) }^{\ast }\left( x^{\ast }\right) \text{%
\qquad }\forall x^{\ast }\in \mathbb{R}^{n}\right\} =C_{0}+0^{+}\left(
C\right) ,
\end{equation*}%
which proves 1.
\end{proof}

\begin{corollary}
\label{char cones}A closed convex set $C\subseteq \mathbb{R}^{n}$ is a cone
with vertex $x_{0}\in C$ (that is, $C-x_{0}$ is a cone) if and only if $%
N_{C}(\mathbb{R}^{n})=N_{C}(x_{0}).$
\end{corollary}

\begin{proof}
Apply Proposition \ref{char Motzk} with $C_{0}:=\left\{ x_{0}\right\} .$
\end{proof}

\bigskip

Another characterization of Motzkin decomposable sets in terms of the range
of its associated normal cone mapping is provided next.

\begin{proposition}
\label{char Motzk 2}A nonempty closed convex set $C\subseteq \mathbb{R}^{n}$
is Motzkin decomposable if and only if%
\begin{equation}
N_{C}(\mathbb{R}^{n})=\left( 0^{+}\left( C\right) \right) ^{0}.
\label{norm = pol}
\end{equation}
\end{proposition}

\begin{proof}
We first observe that the inclusion $\subseteq $ in (\ref{norm = pol}) holds
for every closed convex set $C.$ Assume now that $C$ is Motzkin decomposable
and $C_{0}$ is a compact component of $C,$ and let $x^{\ast }\in 0^{+}\left(
C\right) .$ Since $C_{0}$ is compact, we have $x^{\ast }\in N_{C_{0}}(x)$
for some $x\in C_{0}.$ Then, for every $c\in C_{0}$ and $d\in 0^{+}\left(
C\right) ,$ we have%
\begin{equation*}
\left\langle x^{\ast },c+d\right\rangle =\left\langle x^{\ast
},c\right\rangle +\left\langle x^{\ast },d\right\rangle \leq \left\langle
x^{\ast },x\right\rangle ,
\end{equation*}%
which shows that $x^{\ast }\in N_{C}(x)\subseteq N_{C}(\mathbb{R}^{n}),$
thus completing the proof of (\ref{norm = pol}).

Conversely, assume that (\ref{norm = pol}) holds and define%
\begin{equation*}
C_{0}:=N_{C}^{-1}\left( S^{n-1}\cap \left( 0^{+}\left( C\right) \right)
^{0}\right) .
\end{equation*}%
The mapping\ $N_{C}^{-1}$ is upper semicontinuous, since $N_{C}^{-1}=\left(
\partial \delta _{C}\right) ^{-1}=\partial \delta _{C}^{\ast }$;\ therefore,
as $S^{n-1}\cap \left( 0^{+}\left( C\right) \right) ^{0}$ is compact, $C_{0}$
is compact, too. We will now prove the equality%
\begin{equation}
C=C_{0}+0^{+}\left( C\right) .  \label{C Motzk}
\end{equation}%
The inclusion $\subseteq $ is immediate, since $C_{0}\subseteq C.$ To prove
the opposite inequality, let $x\in C$ and $\overline{x}$ denote the
projection of $x$ onto $C_{0}+0^{+}\left( C\right) .$ Then, using (\ref{norm
= pol}), we obtain $x-\overline{x}\in N_{C_{0}+0^{+}\left( C\right) .}\left( 
\overline{x}\right) \subseteq \left( 0^{+}\left( C\right) \right) ^{0}=N_{C}(%
\mathbb{R}^{n}),$ so that%
\begin{equation}
x-\overline{x}\in N_{C}(\widehat{x})\text{ for some }\widehat{x}\in C.
\label{diff normal}
\end{equation}%
Since $\overline{x}\in C,$ by (\ref{diff normal}) we have%
\begin{equation}
\left\langle x-\overline{x},\overline{x}-\widehat{x}\right\rangle \leq 0.
\label{ineq}
\end{equation}%
We claim that $x=\overline{x}.$ Suppose, towards a cotradiction, that $x\neq 
\overline{x}.$ Then%
\begin{equation*}
\frac{x-\overline{x}}{\left\Vert x-\overline{x}\right\Vert }\in S^{n-1}\cap
\left( 0^{+}\left( C\right) \right) ^{0},
\end{equation*}%
which implies that $\widehat{x}\in C_{0}.$ Therefore, $\left\langle x-%
\overline{x},\widehat{x}-\overline{x}\right\rangle \leq 0,$ which, together
with (\ref{ineq}), yields $\left\langle x-\overline{x},\overline{x}%
\right\rangle =\left\langle x-\overline{x},\widehat{x}\right\rangle .$ We
thus have%
\begin{equation*}
\left\Vert x-\overline{x}\right\Vert ^{2}=\left\langle x-\overline{x},x-%
\overline{x}\right\rangle =\left\langle x-\overline{x},x-\widehat{x}%
\right\rangle \leq 0,
\end{equation*}%
the latter inequality following from (\ref{diff normal}). So, we conclude
that%
\begin{equation*}
x=\overline{x}\in C_{0}+0^{+}\left( C\right) ,
\end{equation*}%
which completes the proof of (\ref{C Motzk}). Taking convex hulls in both
sides of (\ref{C Motzk}), we obtain $C=\mathrm{conv}\left( C_{0}+0^{+}\left(
C\right) \right) =\mathrm{conv}\left( C_{0}\right) +0^{+}\left( C\right) .$
Since $\mathrm{conv}\left( C_{0}\right) $ is a compact convex set, this
shows that $C$ is Motzkin decomposable.
\end{proof}

\bigskip

Recall that the recession function $f0^{+}:\mathbb{R}^{n}\longrightarrow 
\mathbb{R\cup }\left\{ +\infty \right\} $ of a l.s.c. proper convex function 
$f:\mathbb{R}^{n}\longrightarrow \mathbb{R\cup }\left\{ +\infty \right\} $
is the l.s.c. proper sublinear function defined by $\mathrm{epi~}%
f0^{+}=0^{+}\left( \mathrm{epi~}f\right) .$ We omit the easy proof of the
following Proposition.

\begin{proposition}
\label{pol epi subl}If $s:\mathbb{R}^{n}\longrightarrow \mathbb{R\cup }%
\left\{ +\infty \right\} $ is a l.s.c. proper sublinear function, then%
\begin{equation*}
\left( \mathrm{epi~}s\right) ^{0}=\left( \left( \mathrm{dom~}s\right)
^{0}\times \left\{ 0\right\} \right) \cup \mathbb{R}_{+}\left( \partial
s\left( 0\right) \times \left\{ -1\right\} \right) .
\end{equation*}
\end{proposition}

Applying Proposition \ref{char Motzk 2} to the epigraph of a l.s.c. proper
convex function yields the following corollary.

\begin{corollary}
\label{char Motzk funct}A l.s.c. proper convex function $f:\mathbb{R}%
^{n}\longrightarrow \mathbb{R\cup }\left\{ +\infty \right\} $ is Motzkin
decomposable if and only if $\left( \mathrm{dom}~f0^{+}\right) ^{0}=N_{%
\mathrm{dom}~f}(\mathbb{R}^{n})$ and $\partial f0^{+}\left( 0\right)
=\partial f(\mathbb{R}^{n}).$
\end{corollary}

\begin{proof}
Acoording to \cite[Lemma 27]{MaLe-Pi-1}, we have%
\begin{equation*}
N_{\mathrm{epi~}f}\left( \mathbb{R}^{n+1}\right) =\left( N_{\mathrm{dom}~f}(%
\mathbb{R}^{n})\times \left\{ 0\right\} \right) \cup \mathbb{R}_{+}\left(
\partial f(\mathbb{R}^{n})\times \left\{ -1\right\} \right) .
\end{equation*}%
Hence, the equivalence easily follows from combining Proposition \ref{char
Motzk 2}, applied to $C:=\mathrm{epi~}f$, with Proposition \ref{pol epi subl}%
, applied to $s:=f0^{+}.$
\end{proof}

\begin{corollary}
A globally Lipschitzian convex function $f:\mathbb{R}^{n}\longrightarrow 
\mathbb{R}$ is Motzkin decomposable if and only if $\partial f0^{+}\left(
0\right) =\partial f(\mathbb{R}^{n}).$
\end{corollary}

\begin{proof}
By \cite[Theorem 10.5]{R70}, the recession function $f0^{+}$ is finite
everywhere; hence, $\left( \mathrm{dom}~f0^{+}\right) ^{0}=\left( \mathbb{R}%
^{n}\right) ^{0}=\left\{ 0\right\} .$ On the other hand, since $f$ is finite
everywhere too, we have $N_{\mathrm{dom}~f}(\mathbb{R}^{n})=N_{\mathbb{R}%
^{n}}(\mathbb{R}^{n})=\left\{ 0\right\} .$ Therefore, the equivalence is an
immediate consequence of Corollary \ref{char Motzk funct}.
\end{proof}

\section{Generalized Minkowski sets}

In this section we introduce the notions of generalized Minkowski set and
generalized Minkowski function and provide several characterizations of such
new notions. In this respect, several auxiliary results are also proved.

\begin{definition}
A nonempty closed convex set $C\subseteq \mathbb{R}^{n}$ which is the convex
hull of its minimal faces is called a \emph{generalized Minkowski set}, that
is, if%
\begin{equation}
C=\mathrm{conv~}MF\left( C\right) .  \label{eq10.07.17.1}
\end{equation}
\end{definition}

\begin{theorem}
\label{th24.02.15.1}Let $C\subseteq \mathbb{R}^{n}$ be a nonempty closed
convex set and $U\subseteq \mathbb{R}^{n}$ be a supplementary subspace to $%
\mathrm{lin}~C$. Then, $C$ is a generalized Minkowski set if and only if $%
C\cap U$ is a Minkowski set.

In particular, $C$ is a generalized Minkowski set if and only if the
orthogonal slice $C\cap (\mathrm{lin}~C)^{\perp }$ is a Minkowski set.
\end{theorem}

\begin{proof}
Assume that $C\cap U$ is a Minkowski set. This quality of $C\cap U$ combined
with \cite[p. 65]{R70} and the additivity of the $\mathrm{conv}$ operator
and Corollary \ref{LDF} leads us to 
\begin{equation*}
\begin{array}{lll}
C=C\cap U+\mathrm{lin}~C & =\mathrm{conv~ext}\left( C\cap U\right) +\mathrm{%
lin}~C &  \\ 
& =\mathrm{conv~ext}\left( C\cap U\right) +\mathrm{conv}~\mathrm{lin}~C & 
\\ 
& =\mathrm{conv}\left( \mathrm{ext}\left( C\cap U\right) +\mathrm{lin}%
~C\right) &  \\ 
& =\mathrm{conv~}MF\left( C\right) . & 
\end{array}%
\end{equation*}

Conversely, assume that $C=\mathrm{conv~}MF\left( C\right) $ and let $x\in
C\cap U$. Since $x\in C$, it follows that $x=\lambda _{1}x_{1}+\cdots
+\lambda _{m}x_{m}$ for some $\lambda _{1},\ldots ,\lambda _{m}\in \lbrack
0,1]$ such that $\lambda _{1}+\cdots +\lambda _{m}=1$ and some $x_{1}\in
F_{1},\ldots ,x_{m}\in F_{m}$, where%
\begin{equation*}
F_{1}=x_{F_{1}}+\mathrm{lin}~C,\ldots ,F_{m}=x_{F_{m}}+\mathrm{lin}~C
\end{equation*}%
are lowest dimensional faces of $C$ and $x_{F_{1}},\ldots ,x_{F_{m}}\in 
\mathrm{ext}\left( C\cap U\right) $ (see Theorem \ref{prop05.06.17.1}).
Consequently, we have $x_{1}=x_{F_{1}}+p_{_{l_{1}}},\ldots
,x_{m}=x_{F_{m}}+p_{_{l_{m}}}$ for some $p_{l_{1}},\ldots ,p_{l_{m}}\in 
\mathrm{lin}~C$ and therefore $x=\lambda _{1}x_{F{1}}+\cdots +\lambda
_{m}x_{F_{m}}+\lambda _{1}p_{_{l_{1}}}+\cdots +\lambda _{m}p_{_{l_{m}}}$,
namely $U{\ni x-\lambda _{1}x_{F_{1}}-\cdots -\lambda _{m}x_{F_{m}}=\lambda
_{1}p_{_{l_{1}}}+\cdots +\lambda _{m}p_{_{l_{m}}}\in \mathrm{lin}~C,}$
which\ shows\ that ${\lambda _{1}p_{_{l_{1}}}+\cdots +\lambda
_{m}p_{_{l_{m}}}}\in U\cap \left( \mathrm{lin}~C\right) =\{0\}$ and\
therefore%
\begin{equation*}
x={\lambda _{1}x_{F{1}}+\cdots +\lambda _{m}x_{F_{m}}}\in \mathrm{conv~ext}%
\left( C\cap U\right) .\ 
\end{equation*}%
Thus,\ the\ inclusion$\ C\cap U\subseteq \mathrm{conv~ext}\left( C\cap
U\right) $ is\ now\ completely\ done,\ and\ the\ opposite\ inclusion\ is\
obvious.
\end{proof}

\bigskip

Recall that the linearity of a l.s.c. proper convex function $f:\mathbb{R}%
^{n}\longrightarrow \mathbb{R\cup }\left\{ +\infty \right\} $ is the linear
subspace%
\begin{equation*}
\mathrm{lin}~f:=\left\{ d\in \mathbb{R}^{n}:f0^{+}\left( -d\right)
=-f0^{+}\left( d\right) \right\} .
\end{equation*}%
The subspace $\mathrm{lin}~f$ consists of those directions in which $f$ is
affine. By \cite[Theorem 4.8]{R70}, the recession function $f0^{+}$ is
linear on $\mathrm{lin}~f$ and, in view of \cite[Theorem 8.8]{R70}, one has%
\begin{equation}
graph~\left( f0^{+}\right) _{|~\mathrm{lin}~f}=\mathrm{lin}~\mathrm{epi~}f.
\label{graph = lin}
\end{equation}%
Taking all this into account, one easily obtains the following corollary.

\begin{corollary}
Let $f:\mathbb{R}^{n}\longrightarrow \mathbb{R\cup }\left\{ +\infty \right\} 
$ be a l.s.c. proper convex function and $U$ be a supplementary subspace to $%
\mathrm{lin}~f.$ Then, $f$ is a generalized Minkowski function if and only
if $f_{|U}$ is a Minkowski function.

In particular, $f$ is a generalized Minkowski function if and only if $%
f_{|\left( \mathrm{lin}~f\right) ^{\bot }}$ is a Minkowski function.
\end{corollary}

\begin{proof}
Observe first that $U\times \mathbb{R}$ is a supplementary subspace to $%
graph~\left( f0^{+}\right) _{|~\mathrm{lin}~f}.$ Hence, by (\ref{graph = lin}%
) and the equality $\left( \mathrm{epi~}f\right) \cap \left( U\times \mathbb{%
R}\right) =\mathrm{epi~}f_{|U},$ it suffices to apply Theorem \ref%
{th24.02.15.1}.
\end{proof}

\bigskip

Another consequence of Theorem \ref{th24.02.15.1} is the next result.

\begin{theorem}
\label{char gM}Let $C\subseteq \mathbb{R}^{n}$ be a nonempty closed convex
set. Then, $C$ is a generalized Minkowski set if and only if there exist a
Minkowski set $C_{0}\subseteq \mathbb{R}^{n},$ a supplementary subspace $%
V\subseteq \mathbb{R}^{n}$ to $\mathrm{aff~}C_{0}-\mathrm{aff~}C_{0}$ and a
linear subspace $L\subseteq V$ such that%
\begin{equation}
C=C_{0}+L.  \label{Mink + subsp}
\end{equation}
\end{theorem}

\begin{corollary}
In particular, $C$ is a generalized Minkowski set if and only if there exist
a Minkowski set $C_{0}\subseteq \mathbb{R}^{n}$ and a linear subspace $%
L\subseteq \left( \mathrm{aff~}C_{0}-\mathrm{aff~}C_{0}\right) ^{\bot }$
such that 
\begin{equation*}
C=C_{0}+L.
\end{equation*}
\end{corollary}

\begin{proof}
If $C$ is a generalized Minkwoski set, then (\ref{Mink + subsp}) holds with%
\begin{equation*}
C_{0}:=C\cap (\mathrm{lin}~C)^{\perp },\text{ }V=\left( \mathrm{aff~}C_{0}-%
\mathrm{aff~}C_{0}\right) ^{\bot }\text{ and }L:=\mathrm{lin}~C.
\end{equation*}%
Indeed, by Theorem \ref{th24.02.15.1}, the set $C\cap (\mathrm{lin}%
~C)^{\perp }$ is generalized Minkowski and, on the other hand, we have%
\begin{equation*}
\mathrm{aff}\left( C\cap (\mathrm{lin}~C)^{\perp }\right) -\mathrm{aff}%
\left( C\cap (\mathrm{lin}~C)^{\perp }\right) \subseteq (\mathrm{lin}%
~C)^{\perp }-(\mathrm{lin}~C)^{\perp }=(\mathrm{lin}~C)^{\perp },
\end{equation*}%
and hence $\mathrm{lin}~C=(\mathrm{lin}~C)^{\perp \bot }\subseteq \left( 
\mathrm{aff~}C_{0}-\mathrm{aff~}C_{0}\right) ^{\bot }.$

Conversely, assume that (\ref{Mink + subsp}) holds for some Minkowski set $%
C_{0}\subseteq \mathbb{R}^{n},$ some supplementary subspace $V\subseteq 
\mathbb{R}^{n}$ to $\mathrm{aff~}C_{0}-\mathrm{aff~}C_{0}$ and some subspace 
$L\subseteq V.$ By applying a translation if necessary, we assume, without
loss of generality, that $0\in C_{0}.$ Then$,$ $\mathrm{aff~}C_{0}-\mathrm{%
aff~}C_{0}=\mathrm{aff~}C_{0}$. We will now prove that%
\begin{equation}
L\subseteq \mathrm{lin}~C\subseteq V.  \label{L = lin}
\end{equation}%
The first inclusion being obvious, we will only prove the second one. To
this aim, let $d\in \mathrm{lin}~C.$ Denoting by $\pi $ the projection
mapping onto $\mathrm{aff~}C_{0}$ corresponding to the direct sum $V\oplus 
\mathrm{aff~}C_{0}=\mathbb{R}^{n},$ we have%
\begin{equation*}
\pi \left( d\right) \in \pi \left( \mathrm{lin}~C\right) \subseteq \mathrm{%
lin}~\pi \left( C\right) =\mathrm{lin}~C_{0}=\left\{ 0\right\} ;
\end{equation*}%
we have here used the equality $\pi \left( C\right) =C_{0}$. Thus, $\pi
\left( d\right) =0,$ that is, $d\in V,$ which completes the proof of (\ref{L
= lin}). We deduce that $\mathrm{lin}~C\cap \mathrm{aff~}C_{0}\subseteq
V\cap \mathrm{aff~}C_{0}=\left\{ 0\right\} ;$ consequently, $\mathrm{aff~}%
C_{0}\subseteq U$ for some supplementary subspace $U$ to $\mathrm{lin}~C.$
We have $C\cap U=\left( C_{0}+L\right) \cap U=C_{0},$ the latter equality
being an easy consequence of the inclusions $C_{0}\subseteq \mathrm{aff~}%
C_{0}\subseteq U$ and $U\cap L\subseteq U\cap \mathrm{aff~}C_{0}\subseteq
U=\left\{ 0\right\} .$ To complete the proof, it suffices to apply Theorem %
\ref{th24.02.15.1}.
\end{proof}

\begin{remark}
The condition that $L\ $is contained in a supplementary space to $\mathrm{%
aff~}C_{0}-\mathrm{aff~}C_{0}$ is essential for the validity of the "if"
statement of Theorem \ref{char gM}, as can be seen by considering, e.g., the
case when $C_{0}$ is the convex hull of a parabola in $\mathbb{R}^{2}$ and $%
L $ is the line trough the origin orthogonal to the axis of the parabola.
\end{remark}

\begin{definition}
A l.s.c. proper convex function $f:\mathbb{R}^{n}\rightarrow \mathbb{R\cup }%
\left\{ +\infty \right\} $ which has a generalized Minkowski epigraph is
called a \emph{generalized Minkowski set.}
\end{definition}

Let $V$ be a subspace of $\mathbb{R}^{n}.$ Recall that a set $S\subseteq 
\mathbb{R}^{n}$ is said to be $V$-invariant if $S+V=S.$

\begin{remark}
\label{inv}Let $U$ be a supplementary subspace to $V.$ One can easily prove
the equivalence%
\begin{equation*}
S\text{ is }V\text{-invariant }\Leftrightarrow \text{ }S\cap U+V=S.
\end{equation*}%
Indeed, if $S$ is $V$-invariant, then $S\cap U+V\subseteq S+V=S,$ and, for
the opposite inclusion, consider an element $s\in S,$ use the direct sum
decomposition%
\begin{equation*}
U\oplus V=\mathbb{R}^{n}
\end{equation*}%
to write%
\begin{equation}
s=s^{U}+s^{V},\text{\qquad with }s^{U}\in U\text{ and }s^{V}\in V,
\label{dir sum}
\end{equation}%
and observe that%
\begin{equation*}
s^{U}=s-s^{V}\in S-V=S+V=S.
\end{equation*}%
Conversely, if $S\cap U+V=S$ and $s\in S$ is decomposed as in (\ref{dir sum}%
), by the uniqueness of this decomposition one has $s^{U}\in S\cap U,$ and
hence%
\begin{equation*}
s+V=s^{U}+s^{V}+V=s^{U}+V\subseteq S\cap U+V=S,
\end{equation*}%
which proves that $S+V=S,$ as the inclusion $\supseteq $ is obvious.

In particular, one has%
\begin{equation*}
S\text{ is }V\text{-invariant }\Leftrightarrow \text{ }S\cap V^{\bot }+V=S.
\end{equation*}
\end{remark}

\begin{lemma}
\label{conv inv}Let $C\subseteq \mathbb{R}^{n}$ be a nonempty closed convex
set and $U\subseteq \mathbb{R}^{n}$ be a supplementary subspace to $\mathrm{%
lin}~C$. If $S\subseteq C$ is a $\left( \mathrm{lin}~C\right) $-invariant
set such that $\mathrm{conv}~S=C,$ then%
\begin{equation}
\mathrm{conv}\left( S\cap U\right) =C\cap U.  \label{inv gen}
\end{equation}

In particular,%
\begin{equation*}
\mathrm{conv}\left( S\cap (\mathrm{lin}~C)^{\perp }\right) =C\cap (\mathrm{%
lin}~C)^{\perp }.
\end{equation*}
\end{lemma}

\begin{proof}
By Remark \ref{inv}, we have%
\begin{eqnarray*}
C\cap U+\mathrm{lin}~C &=&C=\mathrm{conv}~S=\mathrm{conv}\left( S\cap U+%
\mathrm{lin}~C\right) \\
&=&\mathrm{conv}\left( S\cap U\right) +\mathrm{conv}\left( \mathrm{lin}%
~C\right) \\
&=&\mathrm{conv}\left( S\cap U\right) +\mathrm{lin}~C,
\end{eqnarray*}%
and hence we obtain (\ref{inv gen}).
\end{proof}

\begin{lemma}
\label{sect Mink}Let $C\subseteq \mathbb{R}^{n}$ be a nonempty closed convex
set and $U\subseteq \mathbb{R}^{n}$ be a supplementary subspace to $\mathrm{%
lin}~C$. If $C\cap U$ is a Minkowski set, then%
\begin{equation*}
\mathrm{conv~}MF\left( C\right) =C.
\end{equation*}
\end{lemma}

\begin{proof}
By Corollary \ref{LDF}, we have%
\begin{eqnarray*}
\mathrm{conv}~MF\left( C\right) &=&\mathrm{conv}\left( \mathrm{ext}\left(
C\cap U\right) +\mathrm{lin}~C\right) \\
&=&\mathrm{conv}\left( \mathrm{ext}\left( C\cap U\right) \right) +\mathrm{%
conv}~\mathrm{lin}~C \\
&=&C\cap U+\mathrm{lin}~C=C.
\end{eqnarray*}
\end{proof}

\begin{proposition}
For a nonempty closed convex set $C\subseteq \mathbb{R}^{n},$ the following
statements are equivalent:

\begin{enumerate}
\item $C$ is generalized Minkowski.\label{st23.01.15.1}

\item There exists a smallest $\left( \mathrm{lin}~C\right) $-invariant set $%
S\subseteq C$ such that $\mathrm{conv}~S=C$.\label{st23.01.15.2}

\item There exists a minimal $\left( \mathrm{lin}~C\right) $-invariant set $%
S\subseteq C$ such that $\mathrm{conv}~S=C.$\label{st23.01.15.3}
\end{enumerate}

In 2 and 3, one has $S=MF\left( C\right) $.

\label{prop22.12.2022.1}
\end{proposition}

\begin{proof}
1 $\Rightarrow $ 2. Since the closed convex set $C\cap (\mathrm{lin}%
~C)^{\perp }$ is Minkowski, according to Theorem \ref{th24.02.15.1}, it
follows that $\mathrm{ext}\left( C\cap (\mathrm{lin}~C)^{\perp }\right) $
is, according to \cite[Proposition 3.1]{MaLe-Pi}, the smallest set whose
convex hull equals $C\cap (\mathrm{lin}~C)^{\perp }$. By Corollary \ref{LDF}%
, the set $MF\left( C\right) $ is $\left( \mathrm{lin}~C\right) $-invariant.
Let $S\subseteq C$ be a $\left( \mathrm{lin}~C\right) $-invariant set such
that $\mathrm{conv}~S=C$. From Lemma \ref{conv inv}, it follows that%
\begin{equation*}
\mathrm{ext}\left( C\cap (\mathrm{lin}~C)^{\perp }\right) \subseteq S\cap (%
\mathrm{lin}~C)^{\perp },
\end{equation*}
which, by Remark \ref{inv}, implies that%
\begin{equation*}
MF\left( C\right) =\mathrm{ext}\left( C\cap (\mathrm{lin}~C)^{\perp }\right)
+\mathrm{lin}~C\subseteq S\cap (\mathrm{lin}~C)^{\perp }+\mathrm{lin}~C=S.
\end{equation*}%
Since, by Lemma \ref{sect Mink}, we have $\mathrm{conv~}MF\left( C\right)
=C, $ in view of (\ref{eq10.07.17.1}) the set $MF\left( C\right) $ is the
smallest $\left( \mathrm{lin}~C\right) $-invariant subset of $C$ whose
convex hull equals $C$.\newline
2 $\Rightarrow $ 3. Obvious.\newline
3 $\Rightarrow $ 1. By Lemma \ref{conv inv}, we have $\mathrm{conv}(S\cap (%
\mathrm{lin}~C)^{\perp })=C\cap (\mathrm{lin}~C)^{\perp }.$ We will actually
prove that $S\cap (\mathrm{lin}~C)^{\perp }$ is minimal among the $\left( 
\mathrm{lin}~C\right) $-invariant sets whose convex hulls equal $C\cap (%
\mathrm{lin}~C)^{\perp }.$ Indeed, if $\mathrm{conv}~S_{1}=C\cap (\mathrm{lin%
}~C)^{\perp }$ for some $\left( \mathrm{lin}~C\right) $-invariant set $%
S_{1}\subseteq S\cap (\mathrm{lin}~C)^{\perp }$, then 
\begin{eqnarray*}
\mathrm{conv}(S_{1}+\mathrm{lin}~C) &=&\mathrm{conv}~S_{1}+\mathrm{conv}~%
\mathrm{lin}~C=C\cap (\mathrm{lin}~C)^{\perp }+\mathrm{conv}~\mathrm{lin}~C
\\
&=&C\cap (\mathrm{lin}~C)^{\perp }+\mathrm{lin}~C=C
\end{eqnarray*}%
and $S_{1}+\mathrm{lin}~C\subseteq S\cap (\mathrm{lin}~C)^{\perp }+\mathrm{%
lin}~C=S$. But since $S_{1}+\mathrm{lin}~C$ is $\left( \mathrm{lin}~C\right) 
$-invariant and $S$ is a minimal $\left( \mathrm{lin}~C\right) $-invariant
set with $\mathrm{conv}~S=C$, it follows that $S_{1}=S_{1}+\mathrm{lin}~C=S$%
. Thus 
\begin{equation*}
S\cap (\mathrm{lin}~C)^{\perp }=S_{1}\cap (\mathrm{lin}~C)^{\perp }=S_{1},
\end{equation*}%
which proves the minimality of $S\cap (\mathrm{lin}~C)^{\perp }$. According
to \cite[Proposition 3.1]{MaLe-Pi}, it follows that the set $C\cap (\mathrm{%
lin}~C)^{\perp }$ is Minkowski; therefore, by Lemma \ref{sect Mink}, we have 
$\mathrm{conv~}MF\left( C\right) =C$.
\end{proof}

\section{The role of the Pareto like sets in this setting}

In \cite{MaLe-Pi}, we paid some special attention to the closed convex sets $%
\emptyset \neq C\subseteq \mathbb{R}^{n}$ for which $M(C)=\mathrm{rbd}(C)$,
where $M(C)$ stands for the \emph{Pareto like set} of $C:$%
\begin{equation*}
\begin{array}{ccc}
M(C) & := & \{x\in C:(x-0^{+}C)\cap C\subseteq x+0^{+}C\} \\ 
& = & \{x\in C:(x-0^{+}C)\cap C=x+\mathrm{lin}~C\}.%
\end{array}%
\end{equation*}

\begin{remark}
\label{rem24.12.2022.2} $($\cite[Proposition 4.4]{MaLe-Pi}$)$ If $C\subseteq 
\mathbb{R}^{n}$ is a closed convex set, then $M(C)$ is $(\mathrm{lin}~C)$%
-invariant.
\end{remark}

\begin{theorem}
\label{prop221215}Let $C\subseteq \mathbb{R}^{n}$ be a nonempty closed
convex set with the property that $\dim (C\cap (\mathrm{lin}~C)^{\perp
})\geq 2$. If $C$ is not an affine variety and $M(C)=\mathrm{rbd}~C$, then $%
C $ is a generalized Minkowski set.
\end{theorem}

In order to prove Theorem \ref{prop221215}, we need the following lemmas. We
will denote by $\pi $ the projection mapping onto $(\mathrm{lin}~C)^{\perp
}. $

\begin{lemma}
Let $C\subseteq \mathbb{R}^{n}$ be a nonempty closed convex set and $%
U\subseteq \mathbb{R}^{n}$ be a supplementary subspace to $\mathrm{lin}~C$.
If $\pi _{U}$ denotes the projection mapping onto $U$ corresponding to the
direct sum decomposition $U\oplus \mathrm{lin}~C=\mathbb{R}^{n}$, then%
\begin{equation}
\pi _{U}\left( C\right) =C\cap U  \label{proj}
\end{equation}%
and%
\begin{equation}
C=\pi _{U}\left( C\right) +\mathrm{lin}~C.  \label{C from proj}
\end{equation}

In particular, if $\pi $ denotes the orthogonal projection mapping onto $(%
\mathrm{lin}~C)^{\perp },$ then%
\begin{equation*}
\pi \left( C\right) =C\cap (\mathrm{lin}~C)^{\perp }
\end{equation*}%
and%
\begin{equation*}
C=\pi \left( C\right) +\mathrm{lin}~C.
\end{equation*}
\end{lemma}

\begin{proof}
Equality (\ref{proj}) is an immediate consequence of the decomposition%
\begin{equation}
C=C\cap U+\mathrm{lin}~C,  \label{dec}
\end{equation}%
since $\pi _{U}$ is linear. Equality (\ref{C from proj}) follows from (\ref%
{dec}) and (\ref{proj}).
\end{proof}

\begin{corollary}
\label{preim}Let $C\subseteq \mathbb{R}^{n}$ be a nonempty closed convex set
and $U\subseteq \mathbb{R}^{n}$ be a supplementary subspace to $\mathrm{lin}%
~C$. If $\pi _{U}$ denotes the projection mapping onto $U$ corresponding to
the direct sum decomposition $U\oplus \mathrm{lin}~C=\mathbb{R}^{n}$, then%
\begin{equation*}
\pi _{U}^{-1}\left( \pi _{U}\left( C\right) \right) =C.
\end{equation*}

In particular, if $\pi $ denotes the orthogonal projection mapping onto $(%
\mathrm{lin}~C)^{\perp },$ then%
\begin{equation*}
\pi ^{-1}\left( \pi \left( C\right) \right) =C.
\end{equation*}
\end{corollary}

\begin{proof}
The inclusion $\supseteq $ being obvious, we will only prove the opposite
one. Let $x\in \pi _{U}^{-1}\left( \pi _{U}\left( C\right) \right) .$ Then,
in view of (\ref{C from proj}), we have%
\begin{equation*}
x=\pi _{U}\left( x\right) +x-\pi _{U}\left( x\right) \in \pi _{U}\left(
C\right) +\mathrm{lin}~C=C,
\end{equation*}%
which proves the inclusion $\subseteq .$
\end{proof}

\begin{lemma}
If $C\subseteq \mathbb{R}^{n}$ is a nonempty closed convex set, then%
\begin{equation}
\mathrm{aff}\left( C\cap (\mathrm{lin}~C)^{\perp }\right) =\left( \mathrm{%
aff~}C\right) \cap (\mathrm{lin}~C)^{\perp }  \label{aff}
\end{equation}%
and%
\begin{equation}
\mathrm{rbd}(C\cap (\mathrm{lin}~C)^{\perp })+\mathrm{lin}~C=\mathrm{rbd}~C.
\label{rbd}
\end{equation}
\end{lemma}

\begin{proof}
The inclusion $\subseteq $ in (\ref{aff}) being obvious, we will only prove
the opposite one. To this aim, let $x\in \left( \mathrm{aff~}C\right) \cap (%
\mathrm{lin}~C)^{\perp }.$ Since $x\in \mathrm{aff~}C$, it follows that%
\begin{equation}
x=\lambda _{1}x_{1}+\cdots +\lambda _{m}x_{m}  \label{aff comb 1}
\end{equation}%
for some $\lambda _{1},\ldots ,\lambda _{m}\in \mathbb{R}$ such that $%
\lambda _{1}+\cdots +\lambda _{m}=1$ and some $x_{1},\ldots ,x_{m}\in C.$
Applying $\pi $ to both sides of (\ref{aff comb 1}), we obtain%
\begin{equation*}
x=\lambda _{1}\pi \left( x_{1}\right) +\cdots +\lambda _{m}\pi \left(
x_{m}\right) .
\end{equation*}%
Thus, by Lemma \ref{proj}, we have $x\in \mathrm{aff}\left( C\cap (\mathrm{%
lin}~C)^{\perp }\right) ,$ which completes the proof of the inclusion $%
\supseteq $ in (\ref{aff}).

To prove the inclusion $\supseteq $ in (\ref{rbd}), let $x\in \mathrm{rbd}~C$
and, for $r>0,$ take%
\begin{equation*}
x_{r}\in B(x,r)\cap \left( \mathrm{aff}~C\right) \cap \left( \mathbb{R}%
^{n}\setminus C\right) .
\end{equation*}%
By Corollary (\ref{preim}), we have $\pi \left( x_{r}\right) \notin C;$
hence, since $\pi $ is nonexpansive, using (\ref{aff}) we get%
\begin{eqnarray*}
&&%
\begin{array}{c}
\pi \left( x_{r}\right) \in B(\pi \left( x\right) ,r)\cap \left( \mathrm{aff}%
~C\right) \cap (\mathrm{lin}~C)^{\perp }\cap \left( \mathbb{R}^{n}\setminus
C\right)%
\end{array}
\\
&&%
\begin{array}{c}
\text{\qquad \qquad }=B(\pi \left( x\right) ,r)\cap \left( \mathrm{aff}%
~C\right) \cap (\mathrm{lin}~C)^{\perp }\cap \left( \mathbb{R}^{n}\setminus
\left( C\cap (\mathrm{lin}~C)^{\perp }\right) \right)%
\end{array}
\\
&&%
\begin{array}{c}
\text{\qquad \qquad }=B(\pi \left( x\right) ,r)\cap \mathrm{aff}\left( C\cap
(\mathrm{lin}~C)^{\perp }\right) \cap \left( \mathbb{R}^{n}\setminus \left(
C\cap (\mathrm{lin}~C)^{\perp }\right) \right) .%
\end{array}%
\end{eqnarray*}%
This shows that $\pi \left( x\right) \in \mathrm{rbd}(C\cap (\mathrm{lin}%
~C)^{\perp }),$ from which we deduce that%
\begin{equation*}
x=\pi \left( x\right) +x-\pi \left( x\right) \in \mathrm{rbd}(C\cap (\mathrm{%
lin}~C)^{\perp })+\mathrm{lin}~C,
\end{equation*}%
as was to be proved.

To prove the inclusion $\subseteq $ in (\ref{rbd}), it suffices to prove that%
\begin{equation}
\mathrm{rbd}(C\cap (\mathrm{lin}~C)^{\perp })\subseteq \mathrm{rbd}~C,
\label{rbd sect}
\end{equation}%
since $\mathrm{rbd}~C+\mathrm{lin}~C=\mathrm{rbd}~C$ in view of \cite[%
Proposition 4.6]{MaLe-Pi}. To see that (\ref{rbd sect}) holds, just observe
that, for every $x\in \mathrm{rbd}(C\cap (\mathrm{lin}~C)^{\perp })$ and $%
r>0,$ using (\ref{aff}) one obtains%
\begin{eqnarray*}
&&%
\begin{array}{c}
B(x,r)\cap \left( \mathrm{aff}~C\right) \cap \left( \mathbb{R}^{n}\setminus
C\right)%
\end{array}
\\
&&%
\begin{array}{c}
\text{\qquad }\supseteq B(x,r)\cap \left( \mathrm{aff}~C\right) \cap (%
\mathrm{lin}~C)^{\perp }\cap \left( \mathbb{R}^{n}\setminus C\right)%
\end{array}
\\
&&%
\begin{array}{c}
\text{\qquad }=B(x,r)\cap \left( \mathrm{aff}~C\right) \cap (\mathrm{lin}%
~C)^{\perp }\cap \left( \left( \mathbb{R}^{n}\setminus \left( C\cap (\mathrm{%
lin}~C)^{\perp }\right) \right) \right)%
\end{array}
\\
&&%
\begin{array}{c}
\text{\qquad }=B(x,r)\cap \mathrm{aff}\left( C\cap (\mathrm{lin}~C)^{\perp
}\right) \cap \left( \left( \mathbb{R}^{n}\setminus \left( C\cap (\mathrm{lin%
}~C)^{\perp }\right) \right) \right) \neq \emptyset .%
\end{array}%
\end{eqnarray*}%
Now, from the inclusion (\ref{rbd sect}), we deduce that 
\begin{equation*}
\mathrm{rbd}(C\cap (\mathrm{lin}~C)^{\perp })+\mathrm{lin}~C\subseteq 
\mathrm{rbd}~C+\mathrm{lin}~C=\mathrm{rbd}~C,
\end{equation*}%
the latter equality being due to \cite[Proposition 4.6]{MaLe-Pi}.
\end{proof}

\bigskip

\begin{proof}[The proof of Theorem \protect\ref{prop221215}]
By using (\ref{rbd}) and \cite[Proposition 4.5]{MaLe-Pi}, the hypothesis $%
M(C)=\mathrm{rbd}~C$ can be rewritten as 
\begin{equation*}
M(C\cap (\mathrm{lin}~C)^{\perp })+\mathrm{lin}~C=\mathrm{rbd}~(C\cap (%
\mathrm{lin}~C)^{\perp })+\mathrm{lin}~C.
\end{equation*}%
From this inequality, taking into account that the sets $M(C\cap (\mathrm{lin%
}~C)^{\perp })$ and $\mathrm{rbd}~(C\cap (\mathrm{lin}~C)^{\perp })$ are
contained in $(\mathrm{lin}~C)^{\perp },$ we obtain%
\begin{eqnarray*}
M(C\cap (\mathrm{lin}~C)^{\perp }) &=&\left( M(C\cap (\mathrm{lin}~C)^{\perp
})+\mathrm{lin}~C\right) \cap \mathrm{lin}~C)^{\perp } \\
&=&\left( \mathrm{rbd}(C\cap (\mathrm{lin}~C)^{\perp })+\mathrm{lin}%
~C\right) \cap \mathrm{lin}~C)^{\perp } \\
&=&\mathrm{rbd}~(C\cap (\mathrm{lin}~C)^{\perp }).
\end{eqnarray*}%
Observe now that the assumption that $C$ is not an affine variety is
equivalent to the nonemptiness of $\mathrm{rbd}~C,$ since $\mathrm{aff}~C$
is connected. Then, by using \cite[Proposition 4.8]{MaLe-Pi} and Theorem \ref%
{th24.02.15.1}, the statement follows immediately.
\end{proof}

\end{document}